\documentclass[11pt]{amsart}
\usepackage{latexsym,amssymb,amsmath,empheq}
\usepackage{amsthm, dsfont}
\usepackage{tikz-cd}
\usepackage[all]{xy}
\usepackage[short,nodayofweek]{datetime}
\usepackage[active]{srcltx}

\leftmargin  \leftmargini
\def\ZZ{{\mathbb Z}}

\def\RR{{\mathbb R}}
\def\CC{{\mathbb C}}

\def\H{{\mathcal H}}

\def\aaa{{\mathfrak a}}
\def\gg{{\mathfrak g}}
\def\hh{{\mathfrak h}}
\def\kk{{\mathfrak k}}
\def\pp{{\mathfrak p}}
\def\zz{{\mathfrak z}}

\DeclareMathOperator{\GL}{GL}

\DeclareMathOperator{\trace}{tr}
\DeclareMathOperator{\im}{Im}
\DeclareMathOperator{\re}{Re}
\DeclareMathOperator{\Sp}{Sp}
\DeclareMathOperator{\GSp}{GSp}

\DeclareMathOperator{\SL}{SL}
\DeclareMathOperator{\Symm}{Symm}

\theoremstyle{plain}
\newtheorem{thm}{Theorem}[section]
\newtheorem{cor}[thm]{Corollary}
\newtheorem{lem}[thm]{Lemma}
\newtheorem{remark}[thm]{Remark}
\newtheorem{prop}[thm]{Proposition}

\theoremstyle{definition}

\theoremstyle{remark}

\newtheoremstyle{Acknowledgements}
  {}
    {}
     {}
     {}
    {\bfseries}
    {}
     {.5em}
     {\thmname{#1}\thmnumber{ }\thmnote{ (#3)}}
\theoremstyle{Acknowledgements}

\begin{document}
\title[Phantom  projection]{Phantom holomorphic projections arising from Sturm's formula}

\author{Kathrin Maurischat and Rainer Weissauer}
\address{\rm {\bf Kathrin Maurischat}, Mathematisches Institut,
   Heidelberg University, Im Neuenheimer Feld 205, 69120 Heidelberg, Germany }
\curraddr{}
\email{\sf maurischat@mathi.uni-heidelberg.de}
\address{\rm {\bf Rainer Weissauer}, Mathematisches Institut,
   Heidelberg University, Im Neuenheimer Feld 205, 69120 Heidelberg, Germany }
\curraddr{}
\email{\sf weissauer@mathi.uni-heidelberg.de}
\thanks{This work was partially supported by the European Social Fund (Kathrin Maurischat).}

\begin{abstract}
We show the analytic continuation of certain Siegel Poincar\'e series to their critical point for weight three in genus two.
We proof that this continuation posesses a nonhomomorphic part and describe it.
We show that Sturm's  operator also produces a nonhomorphic share for weight three, we call it a phantom term. 
Weight three is the distinguished weight for genus two where this phenomenon arises.
\end{abstract}

\maketitle
\setcounter{tocdepth}{2}
\tableofcontents

\section{Introduction}
For hermitian symmetric domains $\mathcal H=G/K$ the complex structure on $\mathcal H$
corresponds to a decomposition $\mathop{Lie}(G) \otimes_\RR \CC$ into subspaces
$\mathop{Lie}(K) \oplus \pp_+ \oplus \pp_-$. The maximal compact subgroup $K$ of $G$
acts on the subspaces $\pp_{\pm}$, corresponding to the complexified holomorphic resp. antiholomorphic
tangent space of $\mathcal H$ at the origin. For arithmetic subgroups $\Gamma$ of the
group $G$ the Hilbert space $L^2(\Gamma\!\setminus\! G,dg)$ 
decomposes into the discrete spectrum $L^2_{\textrm{dis}}(\Gamma\!\setminus\! G,dg)$ and a continous
spectrum.
Part of this discrete spectrum is obtained as follows. By  assumption the group $G$
is of hermitian type, hence there exist discrete series representations  of $G$.  
Among these are the representations $\pi$ of the holomorphic discrete series.
The isotypical subspace $L^2(\Gamma\!\setminus\! G,dg)_\pi$ of $L^2_{\textrm{dis}}(\Gamma\!\setminus\! G,dg)$
on which $G$ acts by one of these holomorphic discrete series representations $\pi$
is isomorphic to a finite direct sum of $\pi$ with multiplicity say $m(\pi)$. 
It is
well known that these subspaces
$L^2(\Gamma\!\setminus\! G,dg)_\pi$ occur in the cuspidal part of the spectrum.
The projection operator
\begin{equation*}
  P_\pi:  L^2(\Gamma\!\setminus\! G,dg) \longrightarrow L^2(\Gamma\!\setminus\! G,dg)_\pi
\end{equation*}
can be studied by various techniques. Since a holomorphic discrete series representation
contains a unique lowest $K$-type $\tau$, it often suffices to study
the projection operator  
$ P_\pi$ on the $\tau$-isotypic subspace $L^2(\Gamma\!\setminus\! G,dg)_\tau$
of the action of $K$ on $L^2(\Gamma\!\setminus\! G,dg)$.
Classically, the analysis of the projection operator $P_\pi$ is then often
achieved by passing
from functions $f$ in $L^2(\Gamma\!\setminus\! G,dg)_\tau$ to functions $h$ 
on $\mathcal H$, defined by $h(gK)= J_\tau(g) f(g)$ using an explicit cocycle factor $J_\tau$ whose definition
involves  the lowest $K$-type of $\pi$. The functions $h$ on $\mathcal H$ so defined transform under the action of $\gamma\in \Gamma$ with 
respect to the modular transformation property $h(\gamma Z)=J_\tau(\gamma,Z)h(Z)$. The functions $f$ in the subspace $L^2(\Gamma\!\setminus\! G,dg)_\pi$
correspond to the holomorphic functions on $\mathcal H$ with this transformation property. Putting $L^2$-conditions aside (for simplicity), 
the projectors $P_\pi$ thus become  holomorphic projectors.
Classically, they were studied in terms of Fourier expansions using the theory of Poincare series
in \cite{sturm1}, \cite{gross-zagier} for the classical case $G=\SL_2$ and in \cite{panchishkin} for the case of the symplectic groups
$G=\Sp_{m}$. 
\medskip 

For the case $G=\SL_2$ this  becomes more concrete as follows. The holomorphic discrete
series $\pi=\pi_k$ of $\SL_2$ is parametrized by the integers $k\geq 2$, the weight of their lowest 
$K$-type, and it is well known that $m(\pi_k)$ can be identified with the dimension
of the space of cuspidal modular forms $[\Gamma,k]_0$ of weight $k$ with respect to the arithmetic group $\Gamma$ on the complex upper half plane $\mathcal H$. 
Suppose $\Gamma$ contains translations, so that the modular transformation property
$h(\gamma z)= j_k(\gamma,z)h(z)$ implies $h(z)=h(z+ n)$ for some integer $n$
that allows to expand $h(z)$ into a Fourier expansion 
\begin{equation*}
 h(z) \:=\: \sum_{t}\  a(t,y) \exp(2\pi i tz) \:=\: \sum_t \ b(t,y) \exp(2\pi i tx)   \: .
\end{equation*}
For $z=x+iy$ the Fourier coefficients $b(t,y)$ are functions of the imaginary part.
In this special context it has been shown by Sturm \cite{sturm1} for $k>2$ and Gross-Zagier \cite{gross-zagier} for $k=2$
that, up to some explicit normalizing constant $c(k)$ depending only on $k$, the  projector $P_\pi$ for $\pi=\pi_k$ from above corresponds to the following  
holomorphic projector defined on the level of Fourier coefficients by 
\begin{equation*}
b(t,y) \quad \mapsto \quad b_{\textrm{hol}}(t) \:=\: c(k) \cdot \int_0^\infty  b(t,y) \exp(-2\pi ty) (ty)^k \frac{dx~dy}{y^2}  \: .
\end{equation*}
So the holomorphic projection of $h(z)$ is given by the holomorphic modular
form $\sum_t b_{\textrm{hol}}(t)  \exp(2\pi i tz)$.
\medskip

Concerning the higher dimensional cases studied in \cite{panchishkin} and \cite{holproj}, 
it turned out that for the holomorphic discrete series of scalar weight high enough again the holomorphic projectors $P_\pi$ are given
by analogous holomorphic projections on the level of Fourier coefficients generalizing
the Sturm projections from above. One could therefore believe, that this holds quite generally
for all holomorphic discrete series representations of the symplectic group $\Sp_m$.
However, in the higher rank case new phenomena occur. Although the Sturm
projection operators are defined, in general they do not coincide with the corresponding 
projector operators in all cases. This may be the general situation for those holomorphic discrete series whose lowest $K$-type is small. 
For the special case $m=2$ we analyze this in detail.
The holomorphic discrete series $\pi=\pi_k$ of the group $\Sp_2$ are indexed by their lowest $K$-types $k=(k_1,k_2)$ that are given by integers $k_1 \geq k_2 \geq 3$. 
The \lq{smallest\rq}\ case is $k_1=k_2=3$, the case of scalar weight 3. We show
that in this case, as opposed to the cases of weight greater $3$, 
the Sturm projection does not describe the projection operator $P_{\pi}$
for the holomorphic discrete series of scalar weight 3.  

For this we use Poincar\'e series of exponential type of weight $(3,3)$. 
Their meromorphic continuation is established by the methods of~\cite{holproj}.
That is, we use the resolvents of special Casimir operators whose pole behavior  on $L^2(\Gamma\backslash \Sp_2(\RR))$ is described by representation theory.
By a careful inspection of the  spectral components containing the $K$-type  $(3,3)$ nontrivially, the  continuation of the Poincar\'e series is seen to be
analytic in the critical point. More precisely, their spectral locus there  consists of the expected holomorphic discrete series representation $\pi_{(3,3)}^-$ and in addition 
the holomorphic representation $\pi_{(1,1)}^{\mathop {hol}}$ of weight one (Theorem~\ref{lem_diskretes_spek}).
Correspondingly, the Poincar\'e series written as modular forms on $\mathcal H$ have a decomposition
\begin{equation*}
 p_T\:=\: f\:+\: \Delta_+^{[2]} (h_T),
\end{equation*}
where $f_T$ is a holomorphic cuspform of weight $(3,3)$ and $\Delta_+^{[2]} (h_T)$ is the nonholomorphic derivative of a holomorphic modular form of weight $(1,1)$ 
by the Maass operator $\Delta_+^{[2]}$. Both components are nontrivial in general (Theorem~\ref{satz_phantom_proj}). This reflects the property of Sturm's operator.
\begin{thm}
 Let $h$ be a holomorphic cuspform on $\mathcal H$ of weight $(k,k)$. For $\kappa=k+2\geq 4$ Sturm's operator applied to the nonholomorphic image $\Delta_+^{[2]} (h)$ is zero, but for 
 $\kappa=3$ it does not vanish. 
\end{thm}
So Sturm's operator will realize the
projection  $P_{\pi}$ only up to some additional phantom projection $Q_\pi$ to 
a space of phantom components
\begin{equation*}
 L^2(\Gamma\!\setminus\! G,dg)_{\mathop{phan}(\pi)} \subset L^2(\Gamma\!\setminus\! G,dg)
\end{equation*}
defined by  certain representations $\mathop{phan}(\pi)$ associated to $\pi$, which in the case above is $\mathop{phan}(\pi)=\pi_{(1,1)}^{\mathop{hol}}$.

\section{Notation and resolvents}
We follow the notation of \cite{holproj}.
Let $G=\Sp_m(\RR)$ be the symplectic group of genus $m$. 
Apart from section~\ref{diffoperators} we restrict to the case  $m=2$. 
Realize  $G$ as the group of those $g\in M_{m,m}(\RR)$ satisfying $g'Wg=W$ for 
\begin{equation*}
 W=\begin{pmatrix}0&-E_m\\E_m&0\end{pmatrix}.
\end{equation*}
We have the usual action of $G$ on the Siegel halfspace $\H$, for $g=\begin{pmatrix}a&b\\c&d\end{pmatrix}\in G$,
\begin{equation*}
 g Z=(aZ+b)(cZ+d)^{-1} \:.
\end{equation*}
The stabilizer $K$ of $i=iE_m\in\H$ is a maximal compact subgroup of 
$G$. It is isomorphic to the unitary group $U(m)$. We denote by
\begin{displaymath}
 g\:\mapsto \:g i=:Z=X+iY
\end{displaymath}
the obvious isomorphism of $G/K$ to $\H$.
Let $\mathfrak g$ be the Lie algebra of $G$ realized as  $\mathfrak g_\CC\subset M_{2m,2m}(\mathbb C)$ consisting of those $g$ satisfying
$g'W+Wg=0$.
Then $\mathfrak g_\CC=\mathfrak p_+\oplus\mathfrak p_-\oplus\mathfrak k_\CC$,
where $\mathfrak k_\CC$ is the Lie algebra of $K$ given by the matrices satisfying
\begin{equation*}
 \begin{pmatrix}A&S\\-S&A\end{pmatrix}\:, \quad A'=-A\:, \quad S'=S\:,
\end{equation*}
and
\begin{equation*}
 \mathfrak p_\pm=\left\{\begin{pmatrix}X&\pm iX\\\pm iX&-X\end{pmatrix},\quad X'=X\right\}.
\end{equation*}
Let $e_{kl}\in M_{m,m}(\mathbb C)$ be the elementary matrix having entries $(e_{kl})_{ij}=\delta_{ik}\delta_{jl}$ 
and let $X^{(kl)}=\frac{1}{2}(e_{kl}+e_{lk})$.
 The elements $(E_\pm)_{kl}=(E_\pm)_{lk}$ of $\mathfrak p_\pm$ are defined to be those corresponding to $X=X^{(kl)}$, 
$1\leq k,l\leq m$. 
Then $(E_\pm)_{kl}$, $1\leq k\leq l\leq m$ form a basis of $\mathfrak p^\pm$.
A basis of $\mathfrak k_\CC$  is given by $B_{kl}$, for $1\leq k,l\leq m$, where  $B_{kl}$ corresponds 
to $A_{kl}=\frac{1}{2}(e_{kl}-e_{lk})$ and $S_{kl}=\frac{i}{2}(e_{kl}+e_{lk})$.
Let $E_\pm$ be the matrix having entries $(E_\pm)_{kl}$. 
Similarly, let $B=(B_{kl})_{kl}$ be the matrix with entries $B_{kl}$ and let $B^\ast$ be its transpose.
$E_+$, $E_-$, $B$ and $B^\ast$ are matrix valued matrices. 
Formal traces of their formal products, e.g. $\trace(E_+E_-)$, are not invariant under cyclic permutations of their arguments.
The center $\zz_\CC$ of the universal enveloping Lie algebra $\mathfrak U(\mathfrak g_\CC)$
 is generated by $m$ elements.
 The Casimir elements $C_1,C_2$ belong to $\zz_\CC$,
\begin{equation*}
 C_1=\frac{1}{2}(\trace(E_+E_-)+\trace(E_-E_+))+\trace(BB)\:,
\end{equation*}
\begin{eqnarray*}
 C_2&=& \frac{1}{2}\bigl(\trace(E_+E_-E_+E_-)+\trace(E_-E_+E_-E_+)+\trace(B^4)+\trace((B^\ast)^4)\bigr)\\
&&+2\bigl(\trace(E_+E_-BB)+\trace(E_-E_+B^\ast B^\ast)\bigr)\\
&&-\sum_{i,j,k,l}\{(E_+)_{kl},(E_-)_{ij}\}B_{jk}B_{il}\\
&&+\frac{(m+1)^2}{2}(\trace(E_+E_-)+\trace(E_-E_+))\:.
\end{eqnarray*}
For a  smooth representation $\pi$ of $G$ the actions of the Casimir elements restricted to scalar $K$-types $(\kappa,\dots,\kappa)$ are given by
\begin{equation*}
 \pi(C_1)\:=\: \pi(\trace(E_+E_-))-\kappa m(m+1-\kappa)
\end{equation*}
and
\begin{eqnarray*}
 \pi(C_2)&=&\pi(\trace(E_+E_-E_+E_-))+m\kappa^4\\
&&+((m+1)^2-2\kappa (m+1)+2\kappa^2)\bigl(\pi(\trace(E_+E_-))-\kappa m(m+1)\bigr)\:.
\end{eqnarray*}

For genus $m=2$ let
\begin{equation}\label{def_kompakte_CUA}
 \hh_\CC \:=\:\CC B_{11}+\CC B_{22} \:\subset\: \kk_\CC
\end{equation}
be a common Cartan subalgebra  for  $\kk_\CC$ and $\gg_\CC$. 
Let $\Delta^+$ be the set of positive roots for $\hh_\CC$ such that their root spaces belong to $\CC B_{12}+\pp^-$. 
Writing $\Lambda=(\Lambda_1,\Lambda_2)$ for $\Lambda \in\hh_\CC^\ast$, where $\Lambda_j=\Lambda(B_{jj})$, 
these root spaces are 
\begin{eqnarray*}
 \gg_{(1,-1)}=\CC B_{12}, & \gg_{(2,0)}=\CC (E_-)_{11}\:,\\
\gg_{(1,1)}=\CC (E_-)_{12}, & \gg_{(0,2)}=\CC (E_-)_{22}\:.
\end{eqnarray*}
Half the sum of positive root  is
\begin{equation*}
 \delta\:=\:\delta_G\:=\:\frac{1}{2}\sum_{\Lambda\in\Delta^+}\Lambda=(2,1)\:,
\end{equation*}
while $\delta_K=\frac{1}{2}(1,-1)$ is half the sum of positive compact roots.
Applying the linear form $\Lambda$ to the images of $C_1$ and $C_2$ under the Harish-Chandra homomorphism we get
\begin{eqnarray*}
 \Lambda(C_1)\:=\:\Lambda(\gamma(C_1))&=&\Lambda_{1}^2+\Lambda_{2}^2-5\:,\\
\Lambda(C_2)\:=\:\Lambda(\gamma(C_2))&=&\Lambda_{1}^4+\Lambda_{2}^4-17+3\Lambda(C_1)\:.
\end{eqnarray*}
The diagonal subalgebra $\mathfrak a_\CC$  is another Cartan subalgebra, and by choosing Euclidean coordinates $\Lambda=(\Lambda_1,\Lambda_2)$ there
we get an  isometric isomorphism to $\mathfrak h_\CC$. So the above formulas retain valid.
Choosing the system of positive roots correspondingly,
\begin{equation*}
 \Delta^+:=\{\alpha_1=(0,2),\alpha_2=(1,-1),\alpha_1+\alpha_2,\alpha_1+2\alpha_2\} \subset \mathfrak a_\CC^\ast\:,
\end{equation*}
$\mathfrak a$ is the  split component of the Borel subgroup 
\begin{equation*}
 B\:=\:\left\{\begin{pmatrix} T&X\\0&T'^{-1}\end{pmatrix}\mid T \textrm{ upper triangular }\right\}\:\subset\: G\:.
\end{equation*}
The Weyl group $W$ of $G$ acts on $\aaa_\CC^\ast$. It is generated by the simple reflections $s_{\alpha_1}$ and $s_{\alpha_2}$. 
We also define $\mathfrak a_1=\ker(\alpha_1)$ to be the split component of the Klingen parabolic $P_1\supset B$, and
$\mathfrak a_2=\ker(\alpha_2)$ to be the split component of the Siegel parabolic $P_2\supset B$.

Let $u$ and $v$ be complex variables. In \cite[sec.~3]{holproj} there are chosen  elements
\begin{eqnarray*}
 D_+(u,\Lambda) &=&  \prod_{\alpha\in\Delta\textrm{ long}} \bigl(\check{\alpha}(\Lambda)-u\bigr)\:,\\
      D_-(v,\Lambda)  &=& \prod_{\alpha\in\Delta\textrm{ short}} \bigl(\check{\alpha}(\Lambda)-v\bigr)\:,
 \end{eqnarray*}
 or equivalently,
\begin{eqnarray*}
 D_+(u,\Lambda)&=& (\Lambda_1^2-u^2)(\Lambda_2^2-u^2)\:,\\
 D_-(v,\Lambda)&=& \bigl((\Lambda_1+\Lambda_2)^2-v^2\bigr)\bigl((\Lambda_1-\Lambda_2)^2-v^2\bigr)\:.
\end{eqnarray*}
They are the images of the Casimir elements
\begin{eqnarray*}
 D_+(u)&:=& \frac{1}{2}\bigl(C_1^2-C_2+11C_1-2(u^2-1)C_1+2(u^2-1)(u^2-4)\bigr)\:,\label{Dplus_in_koordinaten}\\
D_-(v)&:=& 2C_2-C_1^2-34C_1-2(v^2-9)C_1+(v^2-9)(v^2-1)\:\label{Dminus_in_koordinaten}
\end{eqnarray*}
under the Harish-Chandra homomorphism.
Let $\Gamma$ be any subgroup of finite index in the full modular group 
$\Sp_2(\ZZ)$ containing 
  the group 
\begin{equation*}
 \Gamma_\infty\:=\:\{\begin{pmatrix}
  \pm E_2&\ast\\0&\pm E_2 
 \end{pmatrix}\in\Sp_2(\ZZ)\}
\end{equation*}
of translations.
The space $L^2(\Gamma\backslash G)$ is a representation space for $G=\Sp_2(\RR)$ by right translations. This $G$-action
 comes along with an action of the universal enveloping algebra 
$\mathfrak U(\gg_\CC)$ on $\mathcal C^\infty$-vectors and action of the elements $D_+(u)$ and $D_-(v)$ allows extension to $L^2(\Gamma\backslash G)$.
Behavior and existence of the resolvents $R_\pm$ of the Casimir operators $D_\pm$ on $L^2(\Gamma\backslash G)$ are regulated by their spectrum.
We are interested in the scalar $K$-type $\kappa=(3,3)$.
\begin{prop}\label{prop_resolventen}
The resolvent $R_-(v)$ exists as a meromorphic function on $\re v >1$.
The resolvent $R_+(u)$ exists as a meromorphic function on $\re u>\frac{1}{2}$. Its spectral pole locus within $L^2(\Gamma\backslash G)_{(3,3)}$ at $u=1$
is given by the two discrete parameters $\Lambda=(2,1)$ and $\Lambda=(0,1)$. 
\end{prop}
\begin{proof}[Proof of Proposition~\ref{prop_resolventen}]
 The meromorphicity of the resolvents on the given domains is shown in~\cite[sec.~3]{holproj}.
 The poles of $R_+(u)$ in $u=1$ are given by the $1$-dimensional continuous spectral component $K_{\alpha_1}(1)=(i\RR,1)$ and discrete components indexed by $\Lambda=(s,1)$.
 By Proposition~\ref{lemma_kein_kont_spek} the first does not occur in $L^2(\Gamma\backslash G)_{(3,3)}$.
 By Theorem~\ref{lem_diskretes_spek}, the  remaining discrete parameters are the claimed.
\end{proof}


\section{On the spectrum of $L^2(\Gamma\backslash G)$}

We give results on the occurrence of the $K$-type $(3,3)$ in $L^2(\Gamma\backslash G)$.

\begin{prop}\label{lemma_kein_kont_spek}
In the $1$-dimensional continuous spectral component included in the parameter $K_{\alpha_1}(1)=(i\RR,1)$ the $K$-type $(3,3)$ is trivial. 
\end{prop}
\begin{proof}[Proof of Proposition~\ref{lemma_kein_kont_spek}] 
By \cite[Sec.~7.1]{konno} this spectral component for the general symplectic group $\tilde G=\mathop{GSp}_2(\mathbb A)$ is  globally given by 
\begin{equation*}
 \int_{i\RR} \mathop{ind}\nolimits_{P_{1}}^G(\omega_1\lvert\cdot\rvert^{it}\otimes \omega(\det))\lvert\det\rvert^{-it/2})~dt\:,
\end{equation*}
for unitary characters $\omega_1,\omega$ of $\mathbb G_m$, where for an element
\begin{equation*}
 m(\lambda,g)\:=\:\begin{pmatrix}\lambda&&&\\
  &a&&b\\
  &&\nu/\lambda&\\
  &c&&d
 \end{pmatrix}
\end{equation*}
of $\tilde M_{Kl}$, where $\lambda\in\mathbb G_m$ and $g=\begin{pmatrix}
                                                          a&b\\
                                                          c&d
                                                         \end{pmatrix}\in\mathop{GL}_2$
 with $\nu=\det g$, the character $ \omega_1\lvert\cdot\rvert^{it}\otimes \omega(\det))\lvert\det\rvert^{-it/2}$
 is defined
 as
 \begin{equation*}
  \omega_1\lvert\cdot\rvert^{it}\otimes \omega(\det))\lvert\det\rvert^{-it/2}(m(\lambda,g))\:=\:
  \omega_1(\lambda)\lvert\lambda\rvert^{it} \omega(\nu)\lvert\nu\rvert^{-it/2}\:.
 \end{equation*}
So at the real place the element
\begin{equation*}
 m_0\:=\:m(1,\begin{pmatrix}
        0&1\\-1&0
       \end{pmatrix})\:\in\: K\cap M_{1}
\end{equation*}
always produces the value
\begin{equation*}
  \omega_1\lvert\cdot\rvert^{it}\otimes \omega(\det))\lvert\det\rvert^{-it/2}(m_0)\:=\: 1\:,
\end{equation*}
while for the $K$-types $\pm(3,3)$ which equal $\det^{\pm3}$ on the unitary group $K\cong U(2)$ we must have (see~\ref{diffoperators})
\begin{equation*}
 \det\nolimits^{\pm3}(\psi(m_0))\:=\: \pm i\:.
\end{equation*}
So this continuous spectral component does not contain $K$-type $(3,3)$.
\end{proof}

\begin{thm}\label{lem_diskretes_spek}
Let $\pi$ be an irreducible unitary representation of $G=\Sp_2(\RR)$ with infinitesimal character in the Weyl group orbit of $\Lambda=(1,s)$ containing the $K$-type $(3,3)$ nontrivially.
Then either $\Lambda=(2,1)$ and $\pi$ is the holomorphic discrete series representation $\pi_{(3,3)}^-$  of minimal weight $(3,3)$, 
or $\Lambda=(0,1)$ and $\pi$ is the   holomorphic representation  $\pi_{(1,1)}^{\mathop{hol}}$ of weight one (non-discrete series).
\end{thm}
Theorem~\ref{lem_diskretes_spek} is proven by the following series of lemmas.

\begin{lem}\label{lem_K-type_infchar}\cite[Theorem~1.1]{zhu}
 If $\pi$ and $\pi'$ are irreducible representations of the same infinitesimal character containing the same scalar $K$-type nontrivially, then $\pi$ and $\pi'$ are equivalent. 
\end{lem}
 We include a simple proof of this lemma.
\begin{proof}[Proof of Lemma~\ref{lem_K-type_infchar}]
It follows from Casselman's subrepresentation theorem that $\pi$ and $\pi'$ are constituents of the same induced representation. By Peter Weyl's theorem a scalar $K$-type occurs with at most 
multiplicity one in an induced representation. So $\pi$ equals $\pi'$ there.
\end{proof}

\begin{lem}\label{nzoukoudi-ausbeute}
 Any irreducible unitary representation $\pi$  of $G$ with infinitesimal character $\Lambda$ Weyl conjugated to $(1,s)$ occurring in $L^2(\Gamma\backslash G)$ 
 is either a discrete series representation or occurs in the following list of irreducible Langlands quotients $J'(P,\sigma,\nu)$ for  parabolic subgroups $P\not=G$.
\begin{enumerate}
\item [(a)] For  the Siegel parabolic subgroup $P=P_2$, $\sigma=\sigma_2^+$ is  the discrete series representation of $M_2$ of minimal $(K\cap M_2)$-type $2$ and $\nu=0$. 
Then $\pi$ has infinitesimal character $\Lambda=(1,-1)$.
 \item [(b)] Let P=$P_{1}$ be the Klingen parabolic subgroup. Either $\sigma=(\sigma_1^\pm,\pm)$ is a discrete series representation of $M_{1}$ of minimal $(K\cap M_{1})$-type
 $1$ and $\nu=e_1$.
 Then $\pi$ has infinitesimal character $\Lambda=(1,0)$. Or $\sigma=(\sigma_2^\pm,-)$ is a discrete series representation of $M_{1}$ of minimal $(K\cap M_1)$-type
 $2$ and $\nu=e_1$.
 Then $\pi$ has infinitesimal character $\Lambda=(1,2)$. 
 \item [(c)] If $P=B$ is the Borel subgroup, $\sigma=1$ is the trivial representation of $M_B$ and $\nu=(1,0)$ or $\nu=(2,1)$.
 Then $\pi$ has infinitesimal character equal to $\nu$. 
\end{enumerate}  
\end{lem}
\begin{proof}[Proof of Lemma~\ref{nzoukoudi-ausbeute}]
 We follow Nzoukoudi's~\cite{nzoukoudi} classification of the irreducible unitary representations via Langlands quotients.
 The parabolic subgroup $P=G$ produces the discrete series.
 
 For the Siegel parabolic, the  quotients $J'(P_2,\sigma,\nu)$ belong to limits discrete series representations $\sigma=\sigma_n^+$ of $M_2\cong\SL_2(\RR)^\pm$ with 
 infintesimal character $(n-1)(e_1-e_2)$ and
 characters $\nu=z(e_1-e_2)$ for complex $z$. For the infinitesimal character $\Lambda=(n+z,-n+z)$ of the quotient   to belong to the Weyl orbit 
 $(1,s)^W$ we must have $z\in \ZZ$. 
 The unitary 
 constraint then is $0\leq z\leq 1$ and $n$ odd. The only possible choice is $\Lambda=(1,-1)$, $\nu=0$ and $\sigma=\sigma_2^+$.
 
 For the Klingen parabolic, $\sigma=(\sigma_n^\pm,\pm)$ is a limit of discrete series of $M_{2}\cong\SL_2(\RR)\times Z_2$ with infinitesimal character 
 $2(n-1)e_2$, and $\nu=2ze_1$, $z$ complex, 
 is the 
 character. So the quotient has infinitesimal character $\Lambda=(2z,2n)$, which belongs to the Weyl orbit $(1,s)^W$ only if $2z=\pm 1$. The unitary constraint then forces
 $\sigma=(\sigma_n^\pm,-)$ for arbitrary $n\geq 0$ or $\sigma=(\sigma_1^\pm,+)$.
 As $J'(P_2,\sigma,\nu)$ isn't discrete series, 
 we have the Eisenstein constraint $\lvert\!\lvert \Lambda\rvert\!\rvert^2\leq \lvert\!\lvert \delta\rvert\!\rvert^2= 5$ for the Langlands quotient to belong to the residual spectrum
 of $L^2(\Gamma\backslash G)$. So $n=0,1$.
 
 The subgroup $M_B$ of the Borel group is isomorphic to $Z_2\times Z_2$, and any representation $\sigma=\sigma^{\epsilon_1,\epsilon_2}$ of $M_B$ is described by two signs 
 $\epsilon_1,\epsilon_2$ on   generators. 
  The infinitesimal character $\Lambda$ of the Langlands quotient equals $\nu=(z_1,z_2)$. 
 The unitary contraint implies that $\nu$ is either purely imaginary, or of the form $(x+iy,x-iy)$ with $0<x\leq \frac{1}{2}$ and $y\in \RR$, 
 or of the form $(x,iy)$ with $0<x\leq 1$ and $y\in \RR^\times$, or $z_1\geq z_2\geq 0$ are real and $z_1+z_2\leq 1$ or $(z_1,z_2)=(2,1)$ and $\sigma=1$.
 So only in the last two cases it may belong to the Weyl orbit $(1,s)^W$.
 In case $\Lambda=\nu=(1,iy)$ the infinitesimal character of the quotient is  not real. So it  doesn't appear in the residual spectrum of  $L^2(\Gamma\backslash G)$.
 The remaining possibilities are $\Lambda=(1,0)$ or $\Lambda=(1,2)$ with $\sigma=1$.
\end{proof}

\begin{lem}\label{discrete_series}
 Among the discrete series of $G$ there is a unique one carrying the  $K$-type $(3,3)$ nontrivially  and having infinitesimal character in the Weyl orbit $(1,s)^W$.
 This is the holomorphic discrete series representation $\pi_3^-$ of minimal $K$-type $(3,3)$.
 \end{lem}
 
 \begin{proof}[Proof of Lemma~\ref{discrete_series}]
 For a  semisimple real Lie group $G$ with $\mathop{rank} G=\mathop{rank} K$ 
 the discrete series representations are parametrized by Harish-Candra parameters (infinitesimal characters) $\Lambda$ 
 which belong the weight lattice and satisfy 
 $\langle \check\alpha ,\Lambda\rangle\not=0$ for all roots $\alpha\in\Delta$ and $\langle \check\alpha ,\Lambda\rangle>0$ for all positive compact roots $\alpha\in\Delta_c^+$.
 There is a unique choice of positive roots $\Delta_\Lambda^+$ for which $\Lambda$ is dominant. 
 Then the Blattner weight for $\Lambda$ is  given by
 \begin{equation*}
  \beta_\Lambda=\frac{1}{2}\sum_{\alpha\in\Delta_\Lambda^+}\alpha - \sum_{\alpha\in\Delta_c^+}\alpha\:.
 \end{equation*}
The minimal $K$-type of $\pi_\Lambda$ is given by the Blattner paramter
\begin{equation*}
 k\:=\: \Lambda +\beta_\Lambda\:,
\end{equation*}
and all other $K$-types of $\pi_\Lambda$ are of the form
\begin{equation*}
 k+\sum_{\alpha\in\Delta_\Lambda^+} n_\alpha \alpha\:,
\end{equation*}
for nonnegative integers $n_\alpha$ (\cite[IX.7]{knapp}).
For $G=\Sp_2(\RR)$ we  choose the Cartan subalgebra $ \mathfrak h_\CC$ of both $G$ and $K$ as in (\ref{def_kompakte_CUA})
Chooinge simple roots $\alpha_1=2\Lambda_2$ and $\alpha_2=\Lambda_1-\Lambda_2$ as before, the short root $\alpha_2$ is compact 
(i.e. its root space belongs $\mathop{Lie}(K)_\CC=\mathfrak k_\CC$) and we choose 
$\Delta_c^+=\{\alpha_2\}$. The root system of $\mathfrak g_\CC$ with respect to $\mathfrak h_\CC$  is
\begin{eqnarray*}
 \Delta&=& \{\pm\alpha_1,\pm\alpha_2,\pm (\alpha_1+\alpha_2),\pm(\alpha_1+2\alpha_2)\}\\
 &=& \{(0,\pm 2),(\pm1,\pm1),(\pm 2,0)\}\:.
\end{eqnarray*}
There are four sectors of weight vectors satisfying the above conditions for $\Delta_c^+$ corresponding to the dominant weights of the following four choices of positive roots.
For the general symplectic group $\GSp_2(\RR)$ these reduce by equivalence to the choices 
\begin{equation*}
 \Delta_1^+\:=\:\{(0,2),(1,-1),(1,1),(2,0)\}\:,
\end{equation*}
where the dominant weights $\Lambda=(\Lambda_1,\Lambda_2)$  satisfy $\Lambda_1>\Lambda_2>0$,
and
 \begin{equation*}
   \Delta_2^+\:=\:\{(0,-2),(1,-1),(1,1),(2,0)\}\:,
 \end{equation*}
for which dominant weights $\Lambda=(\Lambda_1,\Lambda_2)$  satisfy $\Lambda_1>-\Lambda_2>0$.

The holomorphic discrete series $\pi_k^-$ of $\GSp_2(\RR)$ belong to $\Delta_1^+$. Here $\beta_1=(1,2)$ and the minimal $K$-type $k$ is given by $k=(\Lambda_1+1,3)$
for dominant infinitesimal character $(\Lambda_1,1)$. For the $K$-type $l$ to occur we must have that $l$ or $-l$ is contained in
$k+\ZZ_{\geq 0}(0,2)+\ZZ_{\geq 0}(1,-1)$. So the $K$-type $l=(3,3)$ only occurs in $\pi_k^-$ if $k=(3,3)$ with infinitesimal character $(2,1)$.

The nonholomorphic discrete series $\pi_k^+$ belong to $\Delta_2^+$. Then $\beta_2=(1,0)$, the dominant infinitesimal characters are $\Lambda=(\Lambda_1,\Lambda_2)$, where $\Lambda_1>-\Lambda_2>0$.
The Blattner parameter for $\Lambda=(\Lambda_1,-1)$ is $k=(\Lambda_1+1,-1)$.
The arising $K$-types $l$ are of the form $\pm l\in k+\ZZ_{\geq 0}(1,1)+\ZZ_{\geq 0}(0,-2)$. So $l=(3,3)$ can occur as a $K$-type in $\pi_k^+$ at most in case $k=(3,-1)$, 
which has infinitesimal character $(2,-1)$ Weyl conjugated to that of $\pi_{(3,3)}^-$. But this is impossible by Lemma~\ref{lem_K-type_infchar}.
\end{proof}

\begin{lem}\label{lemma_ausschlussverfahren}
 Concerning case (b) of Lemma~\ref{nzoukoudi-ausbeute}, the Langlands quotient belonging to $\sigma=(\sigma_1^-,+)$ with $\nu=e_1$ has  nontrivial $K$-type $(3,3)$.
 It is the  holomorphic representation  $\pi_{(1,1)}^{\mathop{hol}}$ of weight one.
 In all other cases of Lemma~\ref{nzoukoudi-ausbeute} (a)-(c) the $K$-type $(3,3)$ is zero. 
\end{lem}
\begin{proof}[Proof of Lemma~\ref{lemma_ausschlussverfahren}]
The $(K\cap M_S)$-types of discrete series representation $\sigma_2^+$ in case (a) are included in $2+2\ZZ$. By Frobenius reciprocity, 
 the induced representation $\mathop{ind}_{P_S}^G(\sigma_2^+)$
 does not contain the odd scalar $K$-type $(3,3)$, nor does its Langlands quotient.
Concerning the limits of discrete series in case (b), the holomorphic limit of discrete series $\sigma_1^-$ of $\SL_2(\RR)$ contains the $K$-type $3$ nontrivially, as well does the 
limit of discrete series $(\sigma_1^-,+)$ of $M_{Kl}\cong \SL_2(\RR)\times Z_2$. Again by Frobenius reciprocity, the corresponding Langlands quotient contains the $K$-type $(3,3)$ nontrivially.

 The irreducible Langlands quotients are pairwise inequivalent. The quotients left by cases (b) and (c) have  infinitesimal characters  conjugated to either $(2,1)$ or $(1,0)$. 
 If one of them contained a nontrivial $K$-type $(3,3)$, it was equivalent to either $\pi_{(3,3)}^-$ or $\pi_{(1,1)}^{\mathop{hol}}$ by Lemma~\ref{lem_K-type_infchar}.
 \end{proof}


\section{Poincar\'e series for weight three}\label{section_konvergenz}
We define  Poincar\'e series of weight $\kappa=3$.
For complex variables $u$ and $v$ and positive definite $(2,2)$-matrices $T$ with half-integral entries
let
\begin{equation*}
 P_T(g,u,v) =\sum_{\gamma\in \Gamma_\infty\backslash \Gamma} H_T(\gamma g,s_1,s_2)\:\:,
\end{equation*}
where
\begin{equation*}
 H_T(g,s_1,s_2)\:=\: \frac{\exp(2\pi i\trace(T Z))}{J_\kappa(g,i)}\trace(TY)^{s_1}
\det (Y)^{s_2}\:,
\end{equation*}
for $J_\kappa(g,Z)=\det(cZ+d)^\kappa$ for $g=\begin{pmatrix}\ast&\ast\\c&d\end{pmatrix}$,
and
\begin{equation*}
 s_1\:=\:\frac{v-2u-1}{2}\quad\textrm{ and }\quad s_2\:=\:\frac{u-(\kappa-m)}{2}\:=\: \frac{u-1}{2}\:.
\end{equation*}
By~\cite[Cor. 4.4]{holproj}  these Poincar\'e series belong to $L^2(\Gamma\backslash G)$ within their area of convergence
\begin{equation*}
A\:=\: \{(u,v)\in\CC^2\mid \re u>2 \textrm{ and } \re v>5\}\:.
\end{equation*}
The function $H_T(g,s_1,s_2)$ is
nonholomorphic (in the variable $g$) apart from $(s_1,s_2)=(0,0)$. One expects the Poincar\'e series to have analytic continuation to the critical point $(s_1,s_2)=(0,0)$, 
equivalently $(u,v)=(1,3)$, which is  holomorphic with respect to $g$.
By the same method as for case $\kappa=4$ (\cite[Sec.~6]{holproj}) we show that indeed this analytic continuation exists along the line $s_1=0$, but that there is a nonholomorphic share.
\begin{thm}\label{analytische_fortsetzung}
 The Poincar\'e series $P_T(\cdot,u,v)$ admit  meromorphic continuation in $L^2(\Gamma\backslash G)_\kappa$ to the cone 
\begin{equation*}
\{(u,v)\in \CC^2\mid \re u>\frac{1}{2},\: \re v>1\}\:.
\end{equation*}
The poles are contained in a finite number of lines $u=\mathop{const.}$ and $v=\mathop{const.}$.
 The limit
\begin{equation*}
 P_T(\cdot,1,3)\::=\: \lim_{u\to 1}  P_T(\cdot,u,2u+1)
\end{equation*}
exists as a function of $L^2(\Gamma\backslash G)_\kappa$. It has a $\mathcal C^\infty$-representative.
Its nonzero spectral components belong to the isotypical components of irreducible representations with infinitesimal characters $\Lambda=(2,1)$ and $\Lambda=(1,0)$.
\end{thm}
\begin{proof}[Proof of Theorem~\ref{analytische_fortsetzung}]
By abuse of notation, we  omit the dependence on $T$ in our notations, i.e. $P(g,u,v):=P_T(g,u,v)$.
As in \cite{holproj} we use Casimir operators and their resolvents for the continuation.
Their actions depend on the weight $\kappa=3$. For the two Casimir operators $C_1$ and $C_2$ we calulate
\begin{eqnarray*}
C_1( P(g,u,v)) &=& 4(s_1^2+2s_1s_2+2s_2^2+2s_1+3s_2)  P(g,u,v)\\
&& -16\pi (s_1+s_2)P(g,u,v+2)\\
&& -8\det(T) s_1(s_1-1) P(g,u+2,v)\\
&& +32\pi\det(T) s_1P(g,u+2,v+2)\:
\end{eqnarray*}
and
\begin{eqnarray*}
&& C_2( P(g,u,v)) \:=\:\\
&&\hspace*{1cm}\bigl(
17u^4+2v^4-12uv^3+30u^2v^2-36u^3v+15u^2\\
&&\hspace*{6cm}+6v^2-18uv-32\bigr)P(g,u,v)\\
&&\hspace*{1cm}+256\pi^2(s_1+s_2)(s_1+s_2+1)P(g,u,v+4)\\
&&\hspace*{1cm}-128\pi (s_1+s_2)\bigl((s_1+s_2)^2+3(s_1+s_2)+\frac{23}{8}\bigr)P(g,u,v+2)\\
&&\hspace*{1cm}+32\det(T)^2s_1(s_1-1)(s_1-2)(s_1-3)P(g,u+4,v)\\
&&\hspace*{1cm}-256\pi\det(T)^2s_1(s_1-1)(s_1-2)P(g,u+4,v+2)\\
&&\hspace*{1cm}-16\det(T) s_1(s_1-1)(7u^2+3v^2-9uv-u+\frac{7}{2})P(g,u+2,v)\\
&&\hspace*{1cm}+512\pi^2\det(T)^2s_1(s_1-1)P(g,u+4,v+4)\\
&&\hspace*{1cm}-64\pi\det(T) s_1(4u^2-3uv-10u+9v-8)P(g,u+2,v+2)\\
&&\hspace*{1cm}-256\pi^2\det(T) (s_1+s_2)(4s_1+2s_2+1)P(g,u+2,v+4)\:.
\end{eqnarray*}
We used  the computer algebra system Magma to verify these results. 
But for continuation we need to apply  operators which produce Poincar\'e series of better convergence in either $u$ or $v$. These operators are the known $D_+(u)$ and $D_-(v)$, respectively:
\begin{equation*}
 D_+(u)\::=\: \frac{1}{2}\bigl(C_1^2-C_2+11C_1-2(u^2-1)C_1+2(u^2-1)(u^2-4)\bigr)
\end{equation*}
and 
\begin{equation*}
 D_-(v)\::=\: 2C_2-C_1^2-34C_1-2(v^2-9)C_1+(v^2-9)(v^2-1)\:,
\end{equation*}
respectively. We get
\begin{eqnarray}
&& D_+(u) P(g,u,v)\:=\:\label{Omega_groesser_angewandt}\label{gleichung_fuer_D_+}\\
&&\quad\quad\quad+16\det(T)^2s_1(s_1-1)(s_1-2)(s_1-3)P(g,u+4,v)\nonumber\\
&&\quad\quad\quad-128\pi\det(T)^2 s_1(s_1-1)(s_1-2)P(g,u+4,v+2)\nonumber\\
&&\quad\quad\quad+256\pi^2\det(T)^2 s_1(s_1-1)P(g,u+4,v+4)\nonumber\\
&&\quad\quad\quad+8\det(T) s_1(s_1-1)(u+1)(v-2)P(g,u+2,v)\nonumber\\
&&\quad\quad\quad-32\pi\det(T) s_1(6s_1s_2+3s_1+8s_2^2-8)P(g,u+2,v+2)\nonumber\\
&&\quad\quad\quad+64\pi^2\det(T)(v-u-2)(u-2)P(g,u+2,v+4)\:,\nonumber
\end{eqnarray}
respectively,
\begin{eqnarray}
 D_-(v)P(g,u,v)&=& +64\pi^2(v-u)(v-u-2)P(g,u,v+4)\label{Omega_kleiner_angewandt}\label{gleichung_fuer_D_-}\\
&&+32\pi (u-1)(v+1)(v-u-2)P(g,u,v+2)\nonumber\\
&&+128\pi\det(T) s_1(s_1-2)(v+1)P(g,u+2,v+2)\nonumber\\
&&-256\pi^2\det(T) (v-u-2)^2P(g,u+2,v+4)\:.\nonumber
\end{eqnarray}
{\it Step 1. Meromorphic continuation.}
In \cite[sec.~3]{holproj} the spectral poles of the resolvents $R_+(u)$ and $R_-(v)$ of $D_+(u)$ and $D_-(v)$, respectively, were studied.
They exist as meromorphic functions with the following properties.
\begin{prop}\cite[Prop.~3.1, 3.2, 3.3]{holproj}\label{prop_meromorphe_resolventen}
 The resolvent $R_+(u)$ of $D_+(u)$ is meromorphic on $\re u>\frac{1}{2}$. On the $2$-dimensional continuous spectrum, which is included in the parameters $\re \Lambda=0$, it is holomorphic.
 On the $1$-dimensional spectrum it is meromorphic with a finite number of simple poles $u=c$ including $u=1$, which arise from components $K_{\alpha_1}(x)$ or $K_{\alpha_1+2\alpha_2}(c)$.
 On the discrete spectrum $R_+(u)$ is meromorphic with a finite number of poles corresponding to the roots of $(\Lambda_1^2-u^2)(\Lambda_2^2-u^2)$
 for infinitesimal characters $\Lambda=(\Lambda_1,\Lambda_2)$.
 
 The resolvent $R_-(v)$ of $D_-(v)$ is meromorphic on $\re v>1$. On the $2$-di\-men\-sional as well as on the $1$-dimensional continuous spectral components it is holomorphic.
 On the discrete spectrum $R_+(u)$ is meromorphic with a finite number of poles  corresponding to the roots of 
 $((\lambda_1+\Lambda_2)^2-v^2)((\lambda_1-\Lambda_2)^2-v^2)$ for infinitesimal characters $\Lambda=(\Lambda_1,\Lambda_2)$. 
\end{prop}

Iterated application (see~\cite[sec.~4]{holproj}) of the resolvents $R_+(u)$ and $R_-(v)$ to the Poincar\'e series yields their meromorphic continuation as $L^2$-functions
to the largest area on which the resolvents exist, that is to the cone 
\begin{equation*}
\{(u,v)\in \CC^2\mid \re u>\frac{1}{2},\: \re v>1\}\:.
\end{equation*}

{\it Step 2. The $L^2$-limit $P(\cdot,1,3)\::=\: \lim_{u\to 1}  P(\cdot,u,2u+1)$ exists.}
Let
\begin{equation*}
L^2(\Gamma\backslash G)\:=\: L^2_{\re\Lambda=0} (\Gamma\backslash G)\:\bigoplus_{\gamma,c}\:L^2_{\gamma,c}(\Gamma\backslash G)
\:\bigoplus_\Lambda \:L^2_\Lambda(\Gamma\backslash G)
\end{equation*}
be the spectral decomposition of  $L^2(\Gamma\backslash G)$ and denote by
\begin{equation*}
 P_\bullet(\cdot,u,v)
\end{equation*}
the   according spectral components of the Poincar\'e series.
Notice that $(u,v)=(1,3)$ is an inner point of the area of meromorphicity.
We analyze the operator $D_+(u)$.
We choose $v=2u+1$, so the limit series has equation $s_1=0$.
Equation (\ref{Omega_groesser_angewandt}) simplifies on the intersection of $s_1=0$ with the cone of convergence to
\begin{equation*}
D_+(u) P(\cdot,u,2u+1)\:=\:64\pi^2\det(T)(u-1)(u-2)P(g,u+2,2u+5)\:,
\end{equation*}
where $P(\cdot, u+2,2u+5)$ actually is  convergent in $(u,v)=(1,3)$.
As the meromorphic continuation is unique, this  holds everywhere on $s_1=0$.
Equivalently, as $D_+(1)=D_+(u)-(u^2-1)(C_1-(u^2-4))$,
\begin{eqnarray*}
D_+(1) P(\cdot,u,2u+1) &=&
(u^2-1)\bigl(C_1-(u^2-4)\bigr) P(\cdot,u,2u+1)\\
&&+64\pi^2\det(T)(u-1)(u-2)P(\cdot,u+2,2u+5)\:.
\end{eqnarray*}
Thus, 
\begin{eqnarray*}
D_+(1)^n P(\cdot,u,2u+1) &=&
(u^2-1)^n
\bigl(C_1-(u^2-4)\bigr)^n P(\cdot,u,2u+1)\\
&&+(u-1)\mathcal P(\cdot, u)\:,
\end{eqnarray*}
where $\mathcal P(\cdot, u)$  is a symbol for a $\CC[u]$-linear combination of Poincar\'e series which  actually converge in $(u,v)=(1,3)$.
Choosing $n$ greater than the pole order of $P(\cdot,u,v)$ in $(u,v)=(1,3)$, we have
\begin{equation*}
 \lim_{u\to 1}\:\lvert\!\lvert  (u^2-1)^n
\bigl(C_1-(u^2-4)\bigr)^n P(\cdot,u)\rvert\!\rvert\:=\:0
\end{equation*}
as well as
\begin{equation*}
  \lim_{u\to 1}\:\lvert\!\lvert  (u-1)\mathcal P(\cdot, u)\rvert\!\rvert\:=\:0\:.
\end{equation*}
Applying Schwarz' inequality we deduce
\begin{equation*}
  \lim_{u\to 1}\:\lvert\!\lvert   D_+(1)^n P(\cdot,u,2u+1)\rvert\!\rvert^2\:=\:0\:.
\end{equation*}
Written according to the spectral decomposition, 
 \begin{eqnarray*}
 0&=& \sum_\Lambda \lvert D_+(1,\Lambda)\rvert^{2n}\lim_{u\to 1}\lvert\!\lvert  P_\Lambda(\cdot,u,2u+1)\rvert\!\rvert^2\\
&& +\sum_{\gamma,c}\lim_{u\to 1}\lvert\!\lvert  D_+(1)^nP_{\gamma,c}(\cdot,u,2u+1)\rvert\!\rvert^2\\
&&+\lim_{u\to 1}\lvert\!\lvert  D_+(1)^nP_{\re\Lambda=0}(\cdot,u,2u+1)\rvert\!\rvert^2\:.
\end{eqnarray*}
So any single summand is zero:
The limit 
$\lim_{u\to 1}\lvert\!\lvert  P_\Lambda(\cdot,u,2u+1)\rvert\!\rvert$ exits for any discrete parameter $\Lambda$.
 It  is nonzero only if $ D_+(1,\Lambda)=(\Lambda_1^2-1)(\Lambda_2^2-1)$ is zero. 
So the remaining nonzero discrete
components $P_{\Lambda}(\cdot,u,v)$ have parameters $\Lambda=(1,\Lambda_2)$ and are indeed analytically continued
 in $(u,v)=(1,3)$. 
On any continuous component apart from $K_\alpha(1)$, 
the resolvent $R_+(1)$ exists, thus there we have
\begin{eqnarray*}
 \lim_{u\to 1}\lvert\!\lvert  P_{\bullet}(\cdot,u,2u+1)\rvert\!\rvert^2
&=& \lim_{u\to 1}\lvert\!\lvert  R_+(1)^nD_+(1)^nP_{\bullet}(\cdot,u,2u+1)\rvert\!\rvert^2\\
&\leq & \lvert\!\lvert R_+(1)\rvert\!\rvert^{2n}_\bullet\cdot\lim_{u\to 1}
\lvert\!\lvert  D_+(1)^n P_{\bullet}(\cdot,u,2u+1)\rvert\!\rvert^2
=0\:.
\end{eqnarray*}
On the component $K_\alpha(1)$ we have $D_+(u)=(u^2-1)M_+(u)$, where $M_+(u)$ can be parametrized by $(u^2+t^2)$ and is bounded from below by $u^2$.
So from
\begin{equation*}
 0\:=\: \lim_{u\to 1}\lvert\!\lvert  D_+(1)^nP_{\alpha,1}(\cdot,u,2u+1)\rvert\!\rvert  \: =\:\lim_{u\to 1}(u^2-1)^n \lvert\!\lvert M_+(1)^nP_{\alpha,1}(\cdot,u,2u+1)\rvert\!\rvert
\end{equation*}
we deduce as before that $\lvert\!\lvert M_+(1)P_{\alpha,1}(\cdot,u,2u+1)\rvert\!\rvert$ exists.
As the resolvent $M_+^{-1}(1)$ is an operator bounded by $u^{-1}=1$, the limit

\begin{eqnarray*}
 \lim_{u\to 1}\lvert\!\lvert  P_{\alpha,1}(\cdot,u,2u+1)\rvert\!\rvert^2
&=& \lim_{u\to 1}\lvert\!\lvert  M_+(1)^{-n}M_+(1)^nP_{\alpha,1}(\cdot,u,2u+1)\rvert\!\rvert^2\\
&\leq & \lvert\!\lvert M_+^{-1}(1)\rvert\!\rvert^{2n}_{\alpha,1}\cdot\lim_{u\to 1}
\lvert\!\lvert  M_+(1)^n P_{\alpha,1}(\cdot,u,2u+1)\rvert\!\rvert^2
\:
\end{eqnarray*}
exists.
So the limit $P(\cdot,1,3):=\lim_{u\to 1}P(\cdot,u,2u+1)$ exists as an $L^2$-function.
\bigskip

{\it Step 3. The spectral  locus of $P(\cdot,u,v)$ in $(1,3)$.}
We examine the possible  spectral components left from Step~2. 
Within the continuous spectrum, the only remaining component is indexed by $\Lambda=(1,i\RR)=K_{\alpha_1}(1)$. But by Proposition~\ref{lemma_kein_kont_spek}, in this component the $K$-type 
$\kappa=(3,3)$ does not occur. As the $K$-type is passed on the continuation, the  $P(\cdot,1,3)$ has weight $\kappa$. So its continuous spectral component is
identically zero.
The remaining discrete spectral components are indexed by $\Lambda=(\Lambda_1,1)$. By Proposition~\ref{lem_diskretes_spek}, only two components occur within $L^2(\Gamma\backslash G)_\kappa$:
The holomorphic discrete series representation $\pi_\kappa$ of minimal $K$-type $\kappa=(3,3)$ and of infinitesimal character $\Lambda=(2,1)$, and a non-discrete series representation
of infinitesimal character $\Lambda=(0,1)$.

{\it Step 4. There is a $C^\infty$-representative.}
The limit $P(\cdot,1,3)$ is the solution of an elliptic differential equation with $C^\infty$-coefficients. So itself is $C^\infty$ by regularity theory.
\end{proof}

\section{Phantom holomorphic projection}\label{holomorphe_projektion}
\subsection{Differential operators}\label{diffoperators}

Let $\rho$ be an irreducible unitary representation of $U(m)$ on a finite dimensional vector space $V_\rho$. 
By the homomorphism $J:G\to\GL_m(\CC)$, $J(g)=(ci+d)$, we get an isomorphism $\psi=J\mid_K:K\tilde\to U_m(\CC)$. So $\rho(g):=\rho\circ\psi(g)$ is an irreducible unitary representation of $K$.
As $\rho$ is the restriction of an irreducible representation of $\GL_m(\CC)$, the element $J_\rho(g)=\rho(ci+d)\in\mathop{Aut}(V_\rho)$ is well-defined for all $g\in G$.

In the following we make use of formulas developed in \cite[\S 3]{weissauersLN}. The notation there is according to the choice of the isomorphism $K\cong U(m)$
given by the complex conjugate $\bar\psi$ of $\psi$. This implies to work instead of $\rho$ with the representation $\rho(\bar\psi(g))=\rho(\bar g)$ on $V_\rho$, which
is isomorphic to the contragredient representation $\rho^\ast\circ\psi$ of $K$. So we must carfully replace $\rho^\ast$ by $\rho$ in some of the formulas in \cite[\S 3]{weissauersLN}.

For the Siegel upper halfspace $\mathcal H=G/K$ of genus $m$ let $C^\infty(\mathcal H, V_\rho)$ be the space of $C^\infty$-functions  on $\mathcal H$ with values in the space $V_\rho$ of 
the $K$-representation $\rho$. We have the isomorphism \cite[p. 30]{weissauersLN}
\begin{eqnarray*}
 C^\infty(G,V_{\rho\circ\bar\psi})^K&\tilde\to& C^\infty(\mathcal H, V_\rho)\:,\\
 f(g)&\mapsto& J_\rho(g)f(g)\:.
\end{eqnarray*}
Here $C^\infty(G,V_\rho)^K$ is the subspace of $K$-invariant functions in  $C^\infty(G,V_\rho)=C^\infty(G)\otimes V_\rho$, on which $K$ acts by right translations $R_gf(x)=f(xg)$.
By Schur's lemma $C^\infty(G,V_\rho)^K$ is the $K$-isotypical component for the representation $(\rho\circ\bar\psi)^\ast$. But this is isomorphic to $\rho\circ\psi$ and
we can identify $C^\infty(G,V_\rho)^K$ with $C^\infty(G)_\rho$, the $\rho$-isotypical component of $C^\infty(G)$ on which $K$ acts by right translations. So we get isomorphisms
\begin{equation*}
 \phi_\rho\::\:C^\infty(G)_\rho\:\tilde\to\: C^\infty(\mathcal H, V_\rho)\:.
\end{equation*}
The universal envelopping algebra $\mathfrak U(\gg_\CC)$ of the complex Lie algebra $\gg_\CC$ of $G$ acts from the right on $C^\infty(G)$ and this action commutes with the left action of $G$.
The abelian Lie algebra $\pp_+$ (respectively $\pp_-$) can be identified with the holomorphic (respectively antiholomorphic) tangent space of $\mathcal H$ in $Z=iE$.
So $\mathfrak U(\pp_-)$ acts on $C^\infty(G)$ by leftinvariant differential operators.
\begin{lem}\label{pplus_hoechstgewichte}
For the adjoint representation of $K$ on $\mathfrak U(\mathfrak p_+)$ it holds
 \begin{equation*}
  \mathfrak U(\mathfrak p_+)\:\cong\:\Symm^\bullet\Symm^2(\CC^m) \:=\: \bigoplus_{\rho_k} V_{\rho_k}\:,
 \end{equation*}
where $\rho_k$ runs through the irreducible repesentations of $K$ of highest weight $k=(k_1,\dots,k_m)$ for $k_i\in 2\ZZ$ and $k_1\geq k_2\geq\dots\geq k_m\geq 0$.
\end{lem}
\begin{proof}[Proof of Lemma~\ref{pplus_hoechstgewichte}]
This is \cite[Lemma~3]{weissauersLN} with respect to the change from $\bar\psi$ to $\psi$. 
\end{proof}
The representation of $K$ on $ \mathfrak U(\mathfrak p_-)$ is  dual to that on $\mathfrak U(\mathfrak p_+)$.
For any irreducible representation $\rho$ in  $\mathfrak U(\mathfrak p_+)$ there are operators $E_+^\rho$ on the $\rho$-isotypical component of  $\mathfrak U(\mathfrak p_+)$, 
respectively $E_-^\rho$ in $\mathfrak U(\mathfrak p_-)$ on the $\rho^\ast$-isotypical component of $\mathfrak U(\mathfrak p_-)$ \cite[p.43f.]{weissauersLN}. 
They map
a $K$-isotypical subspace $C^\infty(G)_\tau$ of $C^\infty (G)$ to the direct sum of $K$-isotypical subspaces $C^\infty(G)_{\tilde\tau}$, where
\begin{equation*}
 \rho\otimes \tau\:=\:\bigoplus \tilde\tau\:,
\end{equation*}
respectively $\rho^\ast\otimes \tau$ in case $E_-^{\rho}$.
We get the Maass operators $E_+$ respectively $E_-$ by choosing $\rho=\rho_k$ of highest weight $k=(2,0,\dots,0)$.
More generally, we get Maass operators $E_+^{[\mu]}$ respectively $E_-^{[\mu]}$ by choosing 
\begin{equation*}
 k=(2,\dots,2,0,\dots,0)
\end{equation*}
where $\mu$ is  the number of $2$s  occurring.
By the above identifications $\phi_\tau$ and $\phi_{\rho\otimes\tau}=\bigoplus_{\tilde\tau}\phi_{\tilde\tau}$ we have the commutative diagramm
\begin{center}
\begin{tikzcd}
  C^\infty(G)_\tau\arrow{d}{E_+^\rho} \arrow{r}{\phi_\tau} & C^\infty(\mathcal H,V_\tau) \arrow{d}{\Delta_+^\rho} \\
   C^\infty(G)_{\rho\otimes\tau}      \arrow{r}{\phi_{\rho\otimes\tau}} & C^\infty(\mathcal H,V_{\rho\otimes\tau})
\end{tikzcd},
\end{center}
respectively
\begin{center}
\begin{tikzcd}
  C^\infty(G)_\tau\arrow{d}{E_-^\rho} \arrow{r}{\phi_\tau} & C^\infty(\mathcal H,V_\tau) \arrow{d}{\Delta_-^\rho} \\
   C^\infty(G)_{\rho^\ast\otimes\tau}      \arrow{r}{\phi_{\rho^\ast\otimes\tau}} & C^\infty(\mathcal H,V_{\rho^\ast\otimes\tau})
\end{tikzcd}.
\end{center}
An explicit description of the operators $\Delta_+^{[\mu]}$ respectively $\Delta_-^{[\mu]}$ is found in \cite[pp. 33, 44]{weissauersLN}. 
$\Delta_+^{[\mu]}$ acts on $C^\infty(\mathcal H,V_\tau)=C^\infty(\mathcal H)\otimes V_\tau$ with values in $C^\infty(\mathcal H)\otimes V_{(2,\dots,2,0,\dots,0)}\otimes V_\tau$ 
For $h\in C^\infty(\mathcal H)$ and $v\in V_\tau$   it is defined up to a factor $2^\mu$ by
\begin{equation*}
 \Delta_+^{[\mu]}\left(h(Z)\otimes v\right)\:=\:
 (2i)^\mu(\tau\otimes \mathop{det}\nolimits^{\frac{1-\mu}{2}})(Y^{-1})\cdot \partial_Z^{[\mu]}\left((\tau\otimes\mathop{det}\nolimits^{\frac{1-\mu}{2}})(Y) h(Z)\otimes v\right)\:.
\end{equation*}
Here $\partial_Z=\frac{1}{2}(\partial_X-i\partial_Y)$ is the matrix valued operator with components $\frac{1}{2}(1+\delta_{ij})\frac{\partial}{\partial_{Z_{ij}}}$ for the symmetric
matrix $Z=X+iY\in\mathcal H$ with components $Z_{ij}$, and for a matrix $M$ let
$M^{[\mu]}=\bigwedge^\mu(M)$ be the matrix of the $\mu$-th exterior power of $M$, i.e. the $\binom{m}{\mu}\times\binom{m}{\mu}$-matrix of the minors of $M$ of size $\mu$ 
(see~\cite[p. 208ff]{freitag}).

Analogously, $\Delta_-^{[\mu]}$ acts on on $C^\infty(\mathcal H,V_\tau)=C^\infty(\mathcal H)\otimes V_\tau$ with values in $C^\infty(\mathcal H)\otimes V_{(0,\dots,0,-2,\dots,-2)}
\otimes V_\tau$ via
\begin{equation*}
 \Delta_-^{[\mu]}\left(h(Z)\otimes v\right)\:=\: \Delta_-^{[\mu]}\left(h(Z)\right)\otimes v\:
\end{equation*}
for $h\in C^\infty(\mathcal H)$ and $v\in V_\tau$ and is defined up to a factor $2^\mu$ by
\begin{equation*}
 \Delta_-^{[\mu]}\left(h(Z)\right)\:=\:
 (-2i)^\mu Y^{[\mu]}(\mathop{det}\nolimits^{\frac{1-\mu}{2}})(Y^{-1})\cdot \bar\partial_{Z}^{[\mu]}\left((\mathop{det}\nolimits^{\frac{1-\mu}{2}})(Y)\cdot h(Z)\right)\cdot Y^{[\mu]}\:.
\end{equation*}

\subsection{Sturm's operator} 
Let   
$\mathcal Y=\{Y=Y'\in M_m(\RR)\mid Y>0\}$ be the space of  positive definite symmetric matrices and let 
$\mathcal X=\{X=X'\in M_m(\RR)\mid \lvert X_{jk}\rvert\leq \frac{1}{2}, 1\leq j,k\leq m\}$ be a stripe of width one. 
Let the genus $m$ equal two.
Let $F\in \tilde{\mathcal M}_{\kappa}(\Gamma)$ be a (nonholomorphic modular) form of bounded growth (see \cite[2.4]{panchishkin})) of weight $\kappa$.
Let $f(g)=J_\kappa(g,i)^{-1}F(gi)$  be the preimage of $F$ under $\phi_\rho$ for $\rho=(\kappa,\kappa)$. Define correspondingly 
$e_{T}(g)=J_\kappa(g,i)^{-1}\exp(2\pi i\trace(Tgi))$.
Sturm's operator for weight $\kappa$ for  positive definite $T$ is
\begin{equation*}
 F\:\mapsto\:a(T)\::=\:c(\kappa)^{-1}\det(T)^{\kappa-\frac{3}{2}}\int_{\Gamma_\infty \backslash G/K} f(g)\overline{e_T(g)}~dg\:.
\end{equation*}
Here $c(\kappa)=\sqrt \pi(4\pi)^{3-\kappa}\Gamma(\kappa-\frac{3}{2})\Gamma(\kappa-2)$ (\cite[p. 84]{panchishkin}) is chosen such that 
Sturm's operator is the identity on Fourier coefficients of holomorphic forms.
Up to a constant  Sturm's operator is the invariant pairing of $F$ with 
$\exp(2\pi i\trace(TZ))$ of level $\kappa$,
\begin{equation*}
F\:\mapsto\: a(T)\:=\:c(\kappa)^{-1}\det(T)^{\kappa-\frac{3}{2}}\int_{\Gamma_\infty\backslash\mathcal H} F(Z)\exp(-2\pi i\trace(TZ))\det(Y)^{\kappa}dv_Z\:,
\end{equation*}
where  $dv_Z=\frac{dX}{\det(Y)^{3/2}}\frac{dY}{\det(Y)^{3/2}}$ is the invariant measure on $\mathcal H$ such that $dg=dv_Zdk$, and $dY=\prod_{k\leq l}dY_{kl}$ 
as well as   $dX=\prod_{k\leq l}dX_{kl}$. Replacing $F$ by its Fourier expansion
\begin{equation*}
 F(Z)\:=\:\sum_{T=T'}A(T,Y)\exp({2\pi i\trace(T X)})
\end{equation*}
we get
\begin{equation*}
a(T)\:=\:c(\kappa)^{-1}\int_{\mathcal Y} A(T,Y)\exp(-2\pi\trace(TY))\det(TY)^{\kappa-\frac{3}{2}}\frac{dY}{\det(Y)^{\frac{3}{2}}}\:.
\end{equation*}
In case $\kappa\geq 4$ Sturm's operator 
\begin{equation*}
 S: F(Z)\:\mapsto\: \sum a(T)\exp(2\pi\trace(TZ))
\end{equation*}
realizes the orthogonal projection to the holomorphic part of $F$. 
The  Fourier expansion
\begin{equation*}
 \tilde F(Z)\:=\:\sum_{T>0}a(T)e^{2\pi i\trace(T Z)}
\end{equation*}
gives rise to a holomorphic cuspform $\tilde F\in[\Gamma,\kappa]_0$ of weight $\kappa$, and for all $f\in[\Gamma,\kappa]_0$ it holds
\begin{equation*}
 \langle F,f\rangle\: =\: \langle \tilde F,f\rangle\:,
\end{equation*}
where $\langle \cdot,\cdot\rangle$ is the scalar product of the Hilbert space $L^2_\kappa(\Gamma\backslash \mathcal H)$.
This is shown in \cite{holproj} (see also \cite{panchishkin} in case $\kappa\geq 5$) by using Poincar\'e series of weight $\kappa$ defined analogously 
to ours which are holomorphic cuspforms for $\kappa\geq 4$.  

We are interested in the action of Sturm's operator on images of $\Delta_+^{[\mu]}$ in the special case  $\mu=2$. So let $k=(2,2)$ and $\rho_k=\mathop{det}^2$ then
\begin{equation*}
 \Delta_+^{[2]}\::\:C^\infty(\mathcal H_2,V_\tau)\:\to\:C^\infty(\mathcal H_2,V_{\tau\otimes\mathop{det}\nolimits^2})\:,
\end{equation*}
and $\Delta_+^{[2]}\circ\phi_\tau=\phi_{\tau\otimes \mathop{det}\nolimits^2}\circ\det(E_+^{[2]})$,  is explicitly given on $\mathcal H$ by
\begin{equation*}
 \Delta_+^{[2]}(h)(Z)\:=\:(2i)^2(\tau\otimes \mathop{det}\nolimits^{-\frac{1}{2}})(Y^{-1})\mathop{det}(\partial_Z)\left((\tau\otimes\mathop{det}\nolimits^{-\frac{1}{2}})(Y)h(Z)\right)\:.
\end{equation*}
\begin{prop}\label{prop_sturm_operator}
Let  $h\in[\Gamma,k]_0$ be a holomorphic cuspform on $\mathcal H$ of weight $(k,k)$, where $\kappa=k+2$.
Sturm's operator applied to the Fourier coefficients  of the nonholomorphic image $\Delta_+^{[2]}(h)$  is zero in case $\kappa\geq 4$. 
But for $\kappa=3$ it does not vanish.
\end{prop}
\begin{proof}[Proof of Proposition~\ref{prop_sturm_operator}]
We apply $\Delta_+^{[2]}$  to a holomorphic cuspform $h\in[\Gamma,k]_0$ on $\mathcal H_2$ of weight $(k,k)$. So $V_\tau=\CC$ and
\begin{equation*}
 \tilde h(Z)\:=\:\Delta_+^{[2]}(h)(Z)\:=\: -4\det(Y^{-1})^{k-\frac{1}{2}}\mathop{det}(\partial_Z)\left(\det(Y)^{k-\frac{1}{2}}h(Z)\right)
\end{equation*}
is a function on $\mathcal H_2$. 
(This formula for Maass' operator is also due to Shimura~\cite{shimura} and can be found in \cite[3.1]{panchishkin}.)
It has $K$-type $(k+2,k+2)$ and belongs to the automorphic representation generated by the holomorphic cuspform $h$.
Let
\begin{equation*}
 h(Z)\:=\:\sum_{T=T'}a(T)\exp(2\pi i\trace(TZ))
\end{equation*}
be its Fourier expansion. The Fourier coefficients $b(T,Y)$ of the function 
\begin{equation*}
 \tilde h(Z)\:=\: b(T;Y)\exp(2\pi i\trace(TZ))
\end{equation*}
are 
\begin{equation*}
 -4a(T)\det(Y)^{\frac{1}{2}-k}\mathop{det}(\partial_Z)\left(\det(Y)^{k-\frac{1}{2}}\exp(2\pi i\trace(TZ))\right)\cdot \exp(-2\pi i\trace(TZ))\:.
\end{equation*}
Let
\begin{equation*}
 A(T,Y)\:=\: b(T,Y)\exp(-2\pi \trace(TY))
 \end{equation*}
 be the  coefficient  in the Fourier expansion  in $X$ only. For Sturm's formula we study whether the limit
\begin{equation}\label{sturm_limit}
 \lim_{s\to 0}\int_{\mathcal Y}A(T,Y)\exp(-2\pi\trace(TY))\det(TY)^{k+2+s-\frac{3}{2}}\frac{dY}{\det(Y)^{\frac{3}{2}}}
\end{equation}
is zero.
We make use of the Lemmas~\ref{lemma_freitag_1}, \ref{lemma_freitag_2}, and \ref{lemma_gamma_level_2} below.
Applying Lemma~\ref{lemma_freitag_1} to the functions $f(Z)=\det(Y)^{k-\frac{1}{2}}$ and $g(Z)=\exp(2\pi i\trace(TZ))$ we get
\begin{eqnarray*}
 \mathop{det}(\partial_Z)\left(\det(Y)^{k-\frac{1}{2}}\exp(2\pi i\trace(TZ))\right)&=&
  -\frac{1}{4}C_2(k-\frac{1}{2})\det(Y)^{k-\frac{3}{2}}g\\
  && 
  -\frac{i}{2}C_1(k-\frac{1}{2})\det(Y)^{k-\frac{3}{2}}\trace(Y2\pi iT)g\\
   && 
  +\det(Y)^{k-\frac{1}{2}}(2\pi i)^2\det(T)g\:,
\end{eqnarray*}
where $C_2(k-\frac{1}{2})=(k-\frac{1}{2})k$ and $C_1(k-\frac{1}{2})=k-\frac{1}{2}$ by Lemma~\ref{lemma_freitag_2}.
So the Fourier coefficient $b(T,Y)$  is
\begin{equation}\label{derivative_fourier_coeff}
(4\pi)^2a(T)\det(T)\left(
 \frac{k(k-\frac{1}{2})}{(4\pi)^2}\det(TY)^{-1}-\frac{(k-\frac{1}{2})}{4\pi}\det(TY)^{-1}\trace(YT)+1\right)\:.
\end{equation}
By a change of variables
\footnote{The integral is $\int_{\mathcal Y}b(T;Y)\exp(-4\pi\trace(TY))\det(TY)^{k-1/2+s}\frac{dY}{\det(Y)^{3/2}}$, which equals
$(4\pi)^2a(T)\det(T)\int_{\mathcal Y}\left(k(k-1/2)(4\pi)^{-2}\det(TY)^{-1}-(k-1/2)(4\pi)^{-1}\det(TY)^{-1}\trace(TY)+1\right)$ 
$\times \exp(-4\pi\trace(TY))\det(TY)^{k-1/2+s}\frac{dY}{\det(Y)^{3/2}}$.
For the change of variables $Y\mapsto 4\pi T^{\frac{1}{2}}Y T^{\frac{1}{2}}$ this equals
$(4\pi)^{-2(k+1/2+s-1)}a(T)\det(T)\sqrt{\pi}s(s-1/2)\Gamma(s+k-1/2)\Gamma(s+k-1)$}
$Y\mapsto 4\pi T^{\frac{1}{2}}Y T^{\frac{1}{2}}$, for the limit (\ref{sturm_limit}) we have to compute  the  integral 
\begin{equation*}
\int_{\mathcal Y}\left(k(k-\frac{1}{2})\det(Y)^{-1}
-(k-\frac{1}{2})\det(Y)^{-1}\trace(Y)+1\right)\exp(-\trace(Y))\det(Y)^{k+\frac{1}{2}+s}\frac{dY}{\det(Y)^\frac{3}{2}}
\end{equation*}
up to the factor $(4\pi)^{-2(k-\frac{1}{2}+s)}a(T)\det(T)$,
which by Lemma~\ref{lemma_gamma_level_2} is given (up to a factor $\sqrt\pi$) by
\begin{eqnarray*}
 &&k(k-\frac{1}{2})\Gamma(s+k-\frac{1}{2})\Gamma(s+k-1)\\
 &&\hspace*{1cm}-2(k-\frac{1}{2})(s+k-\frac{1}{2})\Gamma(s+k-\frac{1}{2})\Gamma(s+k-1)+\Gamma(s+k+\frac{1}{2})\Gamma(s+k)\:.
\end{eqnarray*}
But this equals
\begin{equation*}
 s(s-\frac{1}{2})\Gamma(s+k-\frac{1}{2})\Gamma(s+k-1)\:.
\end{equation*}
So for all $k>1$ respectively $\kappa=k+2>3$ the limit (\ref{sturm_limit}) is  zero,
\begin{equation*}
 \lim_{s\to 0}\int_{\mathcal Y}A(T,Y)\exp(-2\pi\trace(TY))\det(TY)^{k+2+s-\frac{3}{2}}\frac{dY}{\det(Y)^{\frac{3}{2}}}\:=\:0\:.
\end{equation*}
While in case $k=1$ (i.e. $\kappa=3$) it is a multiple of
\begin{equation*}
 \lim_{s\to 0}s(s-\frac{1}{2})\Gamma(s+\frac{1}{2})\Gamma(s)\:=\:-\frac{1}{2}\Gamma(\frac{1}{2})\:\not=\:0\:.
\end{equation*}
Collecting constants and having in mind $c(3)=\frac{\pi}{2}$, the Sturm operator maps
\begin{equation*}
A(T,Y)\:\mapsto\: -\frac{1}{4\pi}a(T)\det(T)\:
\end{equation*}
in case $k=1$, while it is zero for $k>1$.
\end{proof}
\begin{remark}
In case of weights $(k,k)$ for $k=0, \frac{1}{2}$ the Fourier coefficients (\ref{derivative_fourier_coeff}) vanish.
For $k=\frac{1}{2}$ this follows from the theory of singular modular forms \cite{freitag_singular}, \cite{resnikoff}. 
So for these weights Sturm's operator is supposed to establish the holomorphic projection, too.
\end{remark}
\begin{lem}\label{lemma_freitag_1}
For quadratic matrices $A,B$ define the symbol $2\cdot(A\cap B)$ by
\begin{equation*}
(A+B)^{[2]}\:=\:A^{[2]}+2\cdot (A\cap B)+B^{[2]}\:.
\end{equation*}
Then it holds
 \begin{equation*}
  \partial_Z^{[2]}(fg)\:=\: \partial_Z^{[2]}(f)g+2(\partial_Z(f)\cap\partial_Z(g))+f\partial_Z^{[2]}(g)\:.
 \end{equation*}
Especially for $m=2$, it holds
\begin{equation*}
 2\cdot\left(Y^{-1}\cap T\right)\:=\:\det(Y)^{-1}\trace(YT)\:.
\end{equation*}
\end{lem}
\begin{proof}[Proof of Lemma~\ref{lemma_freitag_1}]
 See~\cite[pp. 208, 211]{freitag}. In case of genus $m=2$ we have $A^{[2]}=\det(A)$, so $2(A\cap B)=A_{11}B_{22}+A_{22}B_{11}-A_{12}B_{21}-A_{21}B_{12}$ and the identity
 $2(Y^{-1}\cap T)=\det(Y)^{-1}\trace(YT)$ follows.
\end{proof}

\begin{lem}\cite[p. 213]{freitag}\label{lemma_freitag_2}
 It holds
 \begin{equation*}
  \partial_Y^{[h]}\det(Y)^\alpha\:=\:C_h(\alpha)\det(Y)^\alpha(Y^{-1})^{[h]}\:,
 \end{equation*}
where $C_h(\alpha)=\alpha(\alpha+\frac{1}{2})\cdots(\alpha+\frac{(h-1)}{2})$.
\end{lem}

\begin{lem}\label{lemma_gamma_level_2}
It holds
 \begin{equation}\label{gamma_level_2}
  \int_{\mathcal Y} e^{-\trace(T Y)}\det(Y)^{s-3/2}~dY\:=\:\sqrt\pi\det(T)^{-s}\Gamma(s)\Gamma(s-\frac{1}{2})
 \end{equation}
as well as 
\begin{equation}\label{gamma_level_2_ableitung}
 \int_{\mathcal Y} e^{-\trace(Y)}\trace(Y)\det(Y)^{s-1}~dY\:=\: 2\sqrt\pi(s+\frac{1}{2})\Gamma(s)\Gamma(s+\frac{1}{2})\:.
\end{equation}
\end{lem}
\begin{proof}[Proof of Lemma~\ref{lemma_gamma_level_2}]
The identity~(\ref{gamma_level_2}) is well-known (e.g.~\cite[p. 467]{shimura}).
Differentiating by $\partial_T$ in $T=E_2$ we get
\begin{equation*}
 \int_{\mathcal Y} e^{-\trace(Y)}Y\det(Y)^{s-3/2}~dY\:=\: sE_2\sqrt\pi\Gamma(s)\Gamma(s-\frac{1}{2})\:,
\end{equation*}
so especially
\begin{equation*}
 \int_{\mathcal Y} e^{-\trace(Y)}\trace(Y)\det(Y)^{s-1}~dY\:=\: 2\sqrt\pi(s+\frac{1}{2})\Gamma(s+\frac{1}{2})\Gamma(s\frac{1}{2})\:.
\end{equation*} 
\end{proof}
\subsection{Poincar\'e series}
Let $\rho$ be the irreducible representation of $K$ of minimal weight $(\kappa,\kappa)$ for $\kappa=3$. The images 
\begin{equation*}
 p_T(g i,s) \:=\: \phi_\rho \left(P_T(g,u,2u+1) \right)\:=\: J_\kappa(g,i)  P_T(g,u,2u+1)\:,
\end{equation*}
of the Poincar\'e series $P(g,u,2u+1)$ under the isomorphism $\phi_\rho$  define Poincar\'e series on $\mathcal H$,
\begin{equation*}
 p_T(Z,s)\:=\:\sum_{\gamma\in\Gamma_\infty\backslash \Gamma}e^{2\pi i\trace(T\gamma Z)}
\frac{\det(\im\gamma Z)^s}{J_\kappa(\gamma,Z)}\:.
\end{equation*}
Here $s$ and $u$ are related by $s=\frac{1}{2}(u-1)$. The Poincar\'e series $p_T$  inherit the analytic properties of their preimages $P_T$.

\begin{cor}
For $\re s+\frac{\kappa}{2}>2$ the series $p_T(z,s)$ converge absolutely and locally uniformly in $s$ and uniformly on the Siegel fundamental domain $\mathcal F$ for $\Gamma$.
They belong to  $L^2_\kappa(\Gamma\backslash \mathcal H)$, the Hilbert space of functions on $\Gamma\backslash \mathcal H$ of weight 
$\kappa$ with scalar product given by
\begin{equation*}
 \langle f,g\rangle \:=\:\int_{\mathcal F} f(Z)\overline{g(Z)}\det(Y)^\kappa~dv_Z\:,
\end{equation*}
where $dv_Z=\det(Y)^{-(m+1)}dXdY$.
They have meromorphic continuations to $\re s>-\frac{1}{2}$ as functions in $L^2_\kappa(\Gamma\backslash \mathcal H)$, which are analytic in the critical point $s=0$. That is,
the limit 
\begin{equation*}\label{poincare_auf_H}
 p_T(\cdot)\::=\:\lim_{s\to 0}p_T(\cdot,s)\:\in\: L^2_\kappa(\Gamma\backslash \mathcal H)
\end{equation*}
 exists in $L^2_\kappa(\Gamma\backslash \mathcal H)$ and is $C^\infty$.
\end{cor}
\begin{thm}\label{satz_phantom_proj}
 The analytic continuations $p_T$ of the Poincar\'e series of weight $(3,3)$ to the critical point $s=0$ decompose into
 \begin{equation*}
 p_T\:=\: f_T\:+\:\Delta_+^{[2]}(h_T)\:,
 \end{equation*}
 where $f_T\in[\Gamma,3]_0$ is a holomorphic cuspform of weight $(3,3)$, and $h_T\in[\Gamma,1]$ is a holomorphic modular form of weight $(1,1)$.
 In general  the two forms $f_T$ and $h_T$ are nonzero. 
 The form $h_T$ is recovered from $p_T$ by application of the antiholomorphic Maass operator $\Delta_-^{[2]}$,
 \begin{equation*}
 \Delta_-^{[2]} (p_T)\:=\: \frac{3}{4}\cdot\: h_T\:.
 \end{equation*}
\end{thm}
\begin{proof}[Proof of Theorem~\ref{satz_phantom_proj}]
The analytic continuations $p_T$ of the Poincar\'e series to $s=0$ have the claimed decomposition by Theorem~\ref{analytische_fortsetzung}.
Let
\begin{equation*}
 F(Z)\:=\:\sum_{T=T'}A(T,Y)e^{2\pi i\trace(T X)}
\end{equation*}
be the Fourier expansion of a nonholomorphic modular form of bounded growth of weight $3$.
We  use   unfolding  to
get for  $\re s+\frac{3}{2}>2$
\begin{eqnarray*}
 \langle F, p_T(\cdot,\bar s)\rangle &=&
\int_{\mathcal F}F(Z)\sum_{\gamma\in\Gamma_\infty\backslash \Gamma}e^{-2\pi i\trace(T\gamma\bar Z)}
\frac{\det(\im\gamma Z)^{s}}{\overline{j(\gamma,Z)}^3}\det(Y)^3~dv_Z\\
&=&
\int_{\mathcal F}\sum_{\gamma\in\Gamma_\infty\backslash \Gamma}F(\gamma Z)e^{-2\pi i\trace(T\gamma\bar Z)}
\det(\im\gamma Z)^{s+3}~dv_Z\\
&=&
\int_{\mathcal X}\int_{\mathcal Y} F(Z)e^{-2\pi i\trace(T\bar Z)}\det(Y)^{s+3}~dv_Z\:.
\end{eqnarray*}
But the last integral 
\begin{eqnarray*}
 &&\int_{\mathcal X}\int_{\mathcal Y} F(Z)e^{-2\pi i\trace(T \bar Z)}\det(Y)^{s+3}~dv_Z\\
&&\hspace*{1cm}=\int_{\mathcal X}\int_{\mathcal Y}\sum_{\tilde T}A(\tilde T,Y)e^{2\pi i\trace((\tilde T-T)X)}
e^{-2\pi\trace(TY)}\det(Y)^{s}~dX~dY\\
&&\hspace*{1cm}=\int_{\mathcal Y} A(T,Y)e^{-2\pi\trace(TY)}\det(Y)^{s}~dY\:
\end{eqnarray*}
exists  for $\re s\geq 0$ (this indeed is the  definition of bounded growth), and its
value at $s=0$  is up to a factor the image $a(T)$
of Sturm's operator applied to $F$.
As analytic continuation  is unique, we have
\begin{equation*}
 \langle F,p_T\rangle\:=\:c(3)\det(T)^{-\frac{3}{2}}a(T)\:.
\end{equation*}
This especially applies to the nonholomorphic function $F=\Delta_+^{[2]}(h)$ for a holomorphic cuspform $h\in[\Gamma,1]_0$, where Sturm's operator is seen to be nonzero  by
Proposition~\ref{prop_sturm_operator}. So the nonholomorphic component  $\Delta_+^{[2]}(h_T)$ cannot be zero in general.
Similarly choosing $F\in[\Gamma,3]_0$, Sturm's operator is the identity on $F$. So the holomorphic component $f_T$ does not vanish in general.

The operator $\Delta_-^{[2]}$ is explicitly given by
\begin{equation*}
 \Delta_-^{[2]} (f(Z))\:=\:(-4)\det(Y)^{5/2}\mathop{det}(\bar\partial_{ Z})\left(\det(Y)^{-1/2}f(Z)\right)\:.
\end{equation*}
Applied to the Poincar\'e series $p_T=f_T+\Delta_+^{[2]} (h_T)$ we have
\begin{equation*}
 \Delta_-^{[2]} (p_T)\:=\: \Delta_-^{[2]}\circ\Delta_+^{[2]} (h_T)\:,
\end{equation*}
as the holomorphic cuspform $f_T$ is deleted by  $\bar\partial_{ Z}$ and thus by $\Delta_-^{[2]}$. We first find for the weight $(1,1)$-function $h_T$
\begin{equation*}
 \Delta_+^{[2]} (h_T)\:=\:\frac{1}{2}\det(Y)^{-1}h_T+2i\det(Y)^{-1}\trace(Y\partial_Z(h_T))-4\mathop{det}(\partial_{ Z})(h_T)\:.
\end{equation*}
Applying $\Delta_-^{[2]}$ to this sum, the second and third term delete each other, while for the first we have
\begin{eqnarray*}
 \Delta_-^{[2]}\left(\frac{1}{2}\det(Y)^{-1}h_T\right)&=& -2\det(Y)^{\frac{5}{2}}\mathop{det}(\bar\partial_{ Z})\left(\det(Y)^{-\frac{3}{2}}h_T\right)\\
 &=&\frac{1}{2}\det(Y)^{\frac{5}{2}} h_T\cdot C_2(-\frac{3}{2})\det(Y)^{-\frac{5}{2}} \:=\:\frac{3}{4}\cdot h_T\:,
\end{eqnarray*}
so
\begin{equation*}
 \Delta_-^{[2]}\circ\Delta_+^{[2]} (h_T)\:=\: \frac{3}{4}\cdot h_T\:.
\end{equation*}
\end{proof}

\end{document}